\documentclass{amsart}
\RequirePackage{amssymb}

\newtheorem{thm}{Theorem}
\newtheorem{lem}{Lemma}

\newtheorem{res}{Corollary}
\newtheorem{rem}{Remark}

\newcommand{\ve}{\varepsilon}
\newcommand{\eps}{\varepsilon}

\renewcommand{\P}{\mathsf P}
\newcommand{\E}{\mathsf E}

\newcommand{\R}{\mathbb R}
\renewcommand{\Re}{\mathbb R}
\newcommand{\df}{ d}
\newcommand{\prt}{\partial}
\newcommand{\mc}[1]{\mathcal {#1}}

\newcommand{\pr}{\mathsf P}

\newcommand{\mb}[1]{\mathbb {#1}}

\newcommand{\be}{\begin{equation}}
\newcommand{\ee}{\end{equation}}
\begin{document}

\title[LAN property for discretely observed solutions to SDE's]
{LAN property for discretely observed solutions to
L\'{e}vy driven SDE's}

\author{D. O. Ivanenko }
\address{Kyiv National Taras Shevchenko  University, Volodymyrska, 64, Kyiv, 01033,
             Ukraine} \email{ida@univ.kiev.ua}

\author{A. M. Kulik}
\address{Institute of Mathematics, Ukrai\-ni\-an National Academy of Sciences,
01601 Tereshchenkivska, 3, Kyiv, Ukraine}
\email{kulik@imath.kiev.ua}

\begin{abstract}
The LAN property is proved in the statistical model based on discrete-time observations of a
solution to a L\'{e}vy driven SDE. The proof is based on a general sufficient condition  for a statistical model based on  a discrete observations of  a Markov process to possess the LAN property, and involves substantially the Malliavin calculus-based integral representations for derivatives of log-likelihood of the model.
\end{abstract}
\maketitle
\section{Introduction}

Consider stochastic equation of the form
\begin{equation}\label{eq1}
    \df X_t^\theta=a_\theta(X_t^\theta)\df t +\df Z_t,
\end{equation}
where $a:\Theta\times\R\to \R$ is a measurable function,
$\Theta=(\theta_1, \theta_2)\in \R$ is a parametric set. For a
given $\theta\in \Theta$, assuming that the drift term $a_\theta$
satisfies the standard local Lipschitz and linear growth
conditions,  Eq. (\ref{eq1}) uniquely defines a Markov process
$X$. The aim of this paper is to establish the \emph{local asymptotic normality} property (LAN in the sequel) in a model, where the process $X$ is discretely observed with a fixed time discretization value $h>0$, and a number of observation $n\to \infty$.

The LAN property provides a convenient and powerful tool for establishing lower efficiency bounds in a statistical model, e.g.  \cite{intro1}, \cite{intro2}, \cite{intro3}. Such a property for statistical models, based on discrete observations of processes with L\'evy noise, was studied mostly in the cases, where the likelihood function (or, at least its ``main part'') is explicit, in a sense, e.g. \cite{AJ07}, \cite{AJ81}, \cite{Hopf97}, \cite{KM13}, \cite{KNT}. In the above references the models are linear in the sense that the process under the observation is either a L\'evy process, or a solution of a linear (Ornstein-Uhlenbeck type) SDE driven by a L\'evy process. The general non-linear case remains non-studied to a great extent, and apparently the main reason for this is that the transition probability density of the observed Markov process in that case is highly implicit. In this paper we develop tools, convenient for proving the LAN property in the framework of discretely observed solutions to SDE's with a L\'evy noise. To make the exposition reasonably transparent, we confine ourselves to a particular case of one-dimensional and one-parameter model, and a fixed sample frequency $h$. Various extensions (general state space, multiparameter model, high frequency sampling, etc.) are visible, but we postpone their detailed analysis for a further research.

Our approach consists of two principal parts. On one hand, we design a general sufficient condition for a statistical model based on  a discrete observations of  a Markov process to possess the LAN property, see Theorem \ref{mainthm1} below. This result extends the classical LeCam's result about the LAN property for i.i.d. samples, and it close  \cite[Theorem 13]{Green}, with some substantial differences in the basic assumptions, which makes our result well designed to a study of a model based on observations of a L\'evy driven SDE, see Remark \ref{r1} below. On the other hand, we integral representations of derivatives of 1st and 2nd orders of the log-likelihood are available: our recent papers \cite{MLE} and \cite{SDer} we have derived such representations using the Malliavin calculus tools. The virtue of this approach is the same with the one developed in \cite{Gobe} in the diffusion setting, but with substantial changes which comes from non-diffusive structure of the noise. Combination of these two principal parts leads to a required LAN property.

The structure of the paper follows the two-stage scheme outlined above.  First we  formulate in Section \ref{s21} (and prove in Section \ref{s3}) a general sufficient condition for the LAN property in a  Markov model. Then we  formulate in Section \ref{s22} (and prove in Section \ref{s32})  our main result about the LAN property for a discretely observed solution to a L\'evy driven SDE; here the proof involves substantially the Malliavin calculus-based integral representations of derivatives of the log-likelihood from \cite{MLE} and \cite{SDer}.

\section{The main results}

\subsection{LAN property for discretely observed Markov processes}\label{s21} Let
 $X$ be a Markov process taking its values in a locally
compact metric space $\mb{X}$. The law of $X$ is assumed to depend on a real-valued parameter $\theta$; in what follows, we assume that the parametric set $\Theta$ is an interval $(\theta_1, \theta_2)\in \R$. We denote by $\P_x^\theta$ the law of $X$ with $X_0=x$, which corresponds to the parameter value $\theta$; the expectation w.r.t. $\P_x^\theta$ is denoted by $\E_x^\theta$.  For a given $h>0$, we denote by $\P_{x,n}^\theta$ the law w.r.t. $\P_{x}^\theta$ of the vector $X^n=\left\{X_{hk}, k=1,\dots ,n\right\}$ of discrete  time
observations  of $X$ with the step $h$.  Denote by $\mc{E}_n$  the statistical experiment generated by the sample  $X^n$ with $X_0=x,$ i.e.
\be\label{stat_exp} \mc{E}_n=\Big(\mb{X}^n, \mathcal{B}(\mb{X}^n),
\P_{x,n}^\theta, \theta\in \Theta\Big); \ee we refer to \cite{IKh}
for the notation and terminology. Our aim is to establish the LAN property for the sequence of experiments $\{\mc{E}_n\}$.

Recall that the sequence of statistical experiments $\{\mc{E}_n\}$ (or, equivalently,  the family $\{\P_{x,n}^\theta, \theta\in\Theta\}$) is said to have \emph{the
LAN property} at the point $\theta_0\in\Theta$ as $n\rightarrow\infty$, if for
some sequence $r(n)>0,n\geq 1$ and all $
u\in\R$

$$ Z_{n,\theta_0}(u):=
\frac{\df\P_{x,n}^{\theta_0+ r(n)u}}{\df\P_{x,n}^{
\theta_0}}(X^n)=
\exp\left\{\Delta_n(\theta_0)u-\frac{1}{2}u^2+\Psi_n(u,\theta_0)\right\},
$$
with \be\label{2} \mc{L}\left(\Delta_n(\theta_0)\ |\ \P_{x,n}^{
\theta_0}\right)\Rightarrow N(0,1), \quad
 n\rightarrow\infty; \ee
\be\label{3} \Psi_n(u,\theta_0)\stackrel{\P_{x,n}^{
\theta_0}}{\longrightarrow}0,\quad n\rightarrow\infty. \ee

 In what follows we assume that $X$ admits a transition probability density $p_h(\theta;x,y)$
w.r.t. some $\sigma$-finite measure $\lambda$. Furthermore, we assume that the experiment
$\mc{E}_1$ is \emph{regular}; that is, for every $x\in \mb{X}$

\begin{itemize} \item[(a)]  the function
$\theta\mapsto p_h(\theta;x,y)$ is continuous  for
$\lambda$-a.a. $y\in \mb{X}$;

\item[(b)] the function  $\sqrt{p_h(\theta;x,\cdot)}$ is
differentiable in $L_2(\mb{X}, \lambda)$; that is, there exists
$q_h(\theta;x,\cdot)\in L_2(\mb{X}, \lambda)$ such that
$$\int_{\mb{X}}\left({\sqrt{p_h(\theta+\delta; x, y)}-\sqrt{p_h(\theta; x,y)}\over
\delta}-q_h(\theta;x,y)\right)^2\lambda(d y)\to 0, \quad \delta\to
0;
$$

\item[(c)] the function $q_h(\theta;x,\cdot)$ is continuous in
$L_2(\mb{X}, \lambda)$ w.r.t. $\theta$; that is,
$$
\int_{\mb{X}}\Big(q_h(\theta+\delta; x, y)-q_h(\theta;
x, y)\Big)^2\lambda(d y)\to 0, \quad \delta\to 0.
$$
\end{itemize}

Denote
\be\label{d_hheta_rep}
g_h(\theta, x,y)=2q_h(\theta;x,y)\sqrt{p_h(\theta;x,y)};
\ee
note that by the definition of $q_h$ the function $g_h$ is well defined and satisfies
\be\label{g_mart}
\E_x^\theta g_h(\theta;x, X_h)=0 \ee for every $x\in \Re,
\theta\in \Theta$. Furthermore, denote
\be\label{fischinf}
I_n(\theta)=\sum_{k=1}^n\E_x^\theta \Big(g_h(\theta;X_{h(k-1)},
X_{hk})\Big)^2=4\E_x^\theta\sum_{k=1}^n\int_{\mb{X}}
\left(q_h(\theta;X_{h(k-1)},y)\right)^2\lambda(\df
y). \ee
Assuming that the statistical experiment $\mc{E}_n$ is regular, the above integral is finite and defines the \emph{Fisher information} for $\mc{E}_n$.

We fix $\theta_0\in \Theta$,  and put $ r(n)=I_n^{-1/2}(\theta_0)$ for $n$ large enough, assuming that for those $n$ one has $I_n(\theta_0)>0$.

\begin{thm}\label{mainthm1} Suppose the following.
\begin{itemize}
\item[1.] Statistical experiment \eqref{stat_exp} is regular for
every $x\in\mb{X}$ and $n\ge1$; for $n$ large enough $I_n(\theta_0)>0$.

\item[2.] The sequence $$r(n)
\sum_{j=1}^ng_h\left(\theta_0;X_{h(j-1)},X_{hj}\right), \quad n\geq 1$$
is asymptotically normal w.r.t. $P_x^{\theta_0}$ with parameters $(0,1)$.

\item[3.] The sequence $$
r^2(n)
\sum_{j=1}^ng_h^2(\theta_0;X_{h(j-1)},X_{hj}), \quad n\geq 1
$$
converges to 1 in $P_x^{\theta_0}$-probability.

\item[4.] There exists a constant $p>2$ such that
\be\label{loc2th1} \lim_{n\rightarrow\infty} r^p(n)\E_x^{\theta_0}
\sum_{j=1}^n\left|g_h(\theta_0;X_{h(j-1)},X_{hj})\right|^p =0. \ee

\item[5.] For every $N>0$ \be\label{loc1th1}
\lim_{n\rightarrow\infty}\sup_{|v|<N}  r^2(n)\E_x^{\theta_0}
\sum_{j=1}^n\int_{\mb{X}}\left(q_h\left(\theta_0+
r(n)v;X_{h(j-1)},y\right)- q_h(\theta_0;X_{h(j-1)},y)\right)^2
\lambda(\df y)=0.\ee

\end{itemize}
Then $\{\P_{x,n}^\theta, \theta\in\Theta\}$ has the LAN property at the point $\theta_0$.
\end{thm}

\begin{rem}\label{r1}
The above theorem is closely related to \cite[Theorem 13]{Green}. One important difference is that in \cite{Green} the main conditions are formulated in the terms of the functions
$$\sqrt{p_h(\theta+t;x,y)/p_h(\theta;x,y)}-1,
$$ while within our approach the main assumptions are imposed on the log-likelihood derivative $g_h(\theta;x,y)$, and can be verified efficiently e.g. in a model where $X$ is defined by an SDE with jumps; see Section \ref{s22} below. Another important difference is that the whole approach in \cite{Green} is developed under the assumption that the log-likelihood function smoothly depends on the parameter $\theta$. For a model where $X$ is defined by an SDE with jumps, such an assumption may be very restrictive, see the detailed discussion in \cite{MLE}. This is the reason why we use instead the assumption of regularity of the experiments, which both is much milder and is easily verifiable, see \cite{MLE}.
\end{rem}

Let us note briefly two possible extensions of the above result, which can be obtained without any essential changes in the proof. We do not expose them here in details, because they will not be used in the current paper.

\begin{rem}  The statement of Theorem \ref{mainthm1} still holds true if $h$ is allowed to
depend on $n$, with conditions 1 -- 5 respectively changed.
\end{rem}

\begin{rem} The statement of Theorem \ref{mainthm1} still holds true if, instead of one $\theta_0$, a sequence $\theta_n\to \theta_0$ is considered, with conditions 2 -- 5 respectively changed. Moreover, in that case relation (\ref{2}) and (\ref{3}) would still hold true if instead of a fixed $u$ a sequence $u_n\to u$ is considered. That is, under the uniform version of conditions 2 -- 5 the
\emph{uniform asymptotic normality} would hold true; see
 \cite[Definition 2.2]{IKh}.
\end{rem}

\subsection{LAN property for families of distributions of solutions to L\'{e}vy driven SDE's}\label{s22}
We assume that $Z$ in the SDE (\ref{eq1}) is  a L\'evy
process without a diffusion component; that is,
$$
Z_t=ct+\int_0^t\int_{|u|>1}u\nu(\df s, \df
u)+\int_0^t\int_{|u|\leq 1}u\tilde\nu(\df s, \df u),
$$
where $\nu$ is a Poisson point measure with the intensity measure
$\df s\mu(\df u)$, and  $\tilde \nu (\df s, \df u)=\nu(\df s, \df
u)-\df s\mu(\df u)$ is respective compensated Poisson measure. In
the sequel, we assume the L\'evy measure $\mu$ to satisfy the
following.

\textbf{H.} (i) for some $\beta>0$,
$$
\int_{|u|\geq 1}u^{4+\beta}\mu(du)<\infty;
$$

(ii) for some $u_0>0$, the restriction of $\mu$ on $[-u_0, u_0]$
has a positive density $m\in C^2\left(\left[-u_0,0\right)\cup
\left(0, u_0\right]\right)$;

(iii) there exists $C_0$ such that
$$
|m'(u)|\leq C_0|u|^{-1}m(u),\quad |m''(u)|\leq C_0u^{-2}m(u),\quad
|u|\in (0, u_0];
$$

(iv)
$$
\left(\log \frac{1}{\eps}\right)^{-1}\mu\Big(\{u:|u|\geq
\eps\}\Big)\to \infty,\quad \eps\to 0.
$$
One particularly important class of  L\'evy processes satisfying
\textbf{H}  consists of \emph{tempered $\alpha$-stable processes}
(see \cite{Ros1}), which arise naturally in models of turbulence
\cite{novikov}, economical models of stochastic volatility
\cite{carr}, etc.

Denote by $C^{k,m}(\Re\times\Theta), k,m\geq 0$  the class of
functions $f:\Re\times\Theta\to \Re$ which has continuous
derivatives
$$
\frac{\prt^i}{\prt x^i}\frac{\prt^j}{\prt\ \theta^j}f, \quad i\leq
k, \quad j\leq m.
$$

About the  coefficient $a_\theta(x)$ in Eq. (\ref{eq1}) we assume the following.

\textbf{A.} (i) $a\in C^{3,2} (\Re\times\Theta)$ have bounded derivatives
$\prt_xa$, $\prt^2_{xx}a$, $\prt^2_{x\theta}a$, $\prt^3_{xxx}a$,
$\prt^3_{x\theta\theta}a$, $\prt^3_{xx\theta}a$,
$\prt^4_{xxx\theta}a$ and
 \be\label{lin_gr1}
|a_\theta(x)|+|\partial_{\theta}
a_\theta(x)|+|\partial^2_{\theta\theta} a_\theta(x)|\leq
    C(1+|x|), \quad \theta\in \Theta, \quad x\in \Re.
\ee

(ii) For a given $\theta_0\in \Theta$, there exists a neighbourhood $(\theta_-, \theta_+)\subset \Theta$ of $\theta_0$ such that
$$
\limsup_{|x|\rightarrow\infty}\frac{a_{\theta}(x)}{x}<0 \quad \hbox{uniformly by } \theta\in (\theta_-, \theta_+).
$$

It is  proved in \cite{MLE}  that, under conditions \textbf{A}(i) and \textbf{H}, the following properties hold:
\begin{itemize}
  \item the Markov process $X$ given by \eqref{eq1} has a
transition probability density $p_t^\theta$ w.r.t. the Lebesgue
measure;
  \item this density has a
derivative $\prt_\theta p_t^\theta(x,y),$ and the statistical
experiment (\ref{stat_exp}) is regular;
  \item the function
$g_t^\theta$, given by \eqref{d_hheta_rep} satisfies
\eqref{g_mart}.
\end{itemize}
Hence all the pre-requisites for Theorem \ref{mainthm1}, given in Section \ref{s21}, are available with  $\lambda(dx)=dx$ (the Lebesgue measure).

Furthermore, under conditions \textbf{A} and \textbf{H}, for $\theta=\theta_0$ corresponding Markov process $X$ is ergodic, i.e. there exists unique invariant probability measure $\varkappa_{inv}^{\theta_0}$ for $X$. One can verify this easily, using  conditions, sufficient for  ergodicity of solutions to L\'evy driven SDE's, given in  \cite{masuda} and \cite{ergod}.
Denote by $\{X^{st, \theta_0}_t, t\in \Re\}$ corresponding stationary version of $X$; that is, a Markov process, defined on whole axis $\Re$, which has the  transition probabilities with $X$ and one-dimensional distributions equal to $\varkappa_{inv}^{\theta_0}$. Clearly, the existence of such a process, on a proper probability space, is guaranteed  by the Kolmogorov consistency theorem. Denote
\be\label{sigma}
\sigma^2(\theta_0)=\E \Big(g_h(\theta_0;X_0^{st, \theta_0},X_h^{st, \theta_0})\Big)^2=\int_{\Re}\int_{\Re}\Big(g_h(\theta_0;x,y)\Big)^2p_h(\theta_0;x,y)\, dy\varkappa_{inv}^{\theta_0}(dx).
\ee

The following theorem performs  the main result of this paper. Its proof is given  in Section \ref{s32} below.

\begin{thm}\label{mainthm2} Let conditions \textbf{A} and \textbf{H} hold true, and
$$
\sigma^2(\theta_0)>0.
$$
Then the family $\{\P_{x,n}^\theta, \theta\in\Theta\}$ possesses  the LAN
property at the point $\theta=\theta_0$.
\end{thm}

\section{Proof of Theorem \ref{mainthm1}}\label{s3}

The method of proof goes back to LeCam's proof of the LAN property for i.i.d. samples, see e.g. Theorem II.1.1 and Theorem II.3.1 in
\cite{IKh}. In the Markov setting, the dependence in the observations lead to some additional technicalities; see e.g. (\ref{ll}). The possible ways to overcome these additional difficulties can be found, in a slightly different setting, in the proof of \cite[Theorem 13]{Green}. In order  to keep the exposition transparent and self-sufficient, we prefer to give a complete proof of Theorem \ref{mainthm1} explicitly, rather than to give a chain of partly relevant references.

We separate the proof into several lemmas; in all the lemmas in this section we assume the conditions of  Theorem \ref{mainthm1} to
be fulfilled. Values $x,\theta_0,$ and $u$ are fixed; we assume that $n$ is large enough, so that $\theta_0+r(n)u\in \Theta$. In order to simplify the notation below we write $\theta$ instead of $\theta_0$.

Denote $$\zeta^\theta_{jn}(u)=\left(\left( \frac{p_h(\theta+
r(n)u;X_{h(j-1)},X_{hj})}{p_h(\theta;X_{h(j-1)},X_{hj})}\right)^{1/2}-1\right)
I\left\{p_h(\theta;X_{h(j-1)},X_{hj})\neq0\right\}.$$

\begin{lem}\label{eta_jn} One has
\be\label{zbch_eta}
\limsup_{n\rightarrow\infty}\sum_{j=1}^{n}\E_x^{\theta}(\zeta^\theta_{jn}(u))^2\leq\frac{1}{4}u^2
\ee and \be\label{zbch_eta2}
\lim_{n\rightarrow\infty}\sum_{j=1}^{n}\E_x^{\theta}\left(\zeta^\theta_{jn}(u)-
\frac{1}{2} r(n)ug_h(\theta;X_{h(j-1)},X_{hj})\right)^2=0. \ee
\end{lem}

\begin{proof}  By  the regularity of  $\mc{E}_1$ and the Cauchy inequality  we have
\begin{multline*}
\E_x^{\theta}\left(\zeta^\theta_{jn}(u)- \frac{1}{2}
r(n)ug_h(\theta;X_{h(j-1)},X_{hj})\right)^2\\= \E_x^{\theta}
\int\limits_{\{y:p^{\theta}_h(z,y)\neq0\}}\left(\sqrt{p_h\left(\theta+
r(n)u;X_{h(j-1)},y\right)}\right.\\\left.-
\sqrt{p_h\left(\theta;X_{h(j-1)},y\right)}- r(n)uq_h(\theta;
X_{h(j-1)}, y)\right)^2\lambda(\df y)\\\leq (
r(n)u)^2\E_x^{\theta}\int_{\mb{R}}\left(\int_0^{1}
q_h\left(\theta+
r(n)uv,X_{h(j-1)},y\right)-q_h\left(\theta;X_{h(j-1)},y\right) \df
v\right)^2\lambda(\df y)\\\leq (
r(n)u)^2\E_x^{\theta}\int_{\mb{R}}\lambda(\df y)\int_0^{1}\left(
q_h\left(\theta+
r(n)uv;X_{h(j-1)},y\right)-q_h\left(\theta;X_{h(j-1)},y\right)\right)^2
\df v.
\end{multline*}
This and \eqref{loc1th1} yield \eqref{zbch_eta2}. To deduce  (\ref{zbch_eta}) from \eqref{zbch_eta2},  recall  an elementary inequality
\be\label{ung}|AB|\leq\frac{\alpha}{2}A^2+\frac{1}{2\alpha}B^2,\quad \alpha>0,\ee
and write
$$
\zeta^\theta_{jn}(u)=
\frac{1}{2} r(n)ug_h(\theta;X_{h(j-1)},X_{hj})+\left(\zeta^\theta_{jn}(u)-
\frac{1}{2} r(n)ug_h(\theta;X_{h(j-1)},X_{hj})\right)=:A+B.
$$
Then
$$\begin{aligned}
\E_x^{\theta}(\zeta^\theta_{jn}(u))^2&\leq (1+\alpha){1\over 4}u^2r^2(n)\E_x^{\theta}\left(g_h(\theta;X_{h(j-1)},X_{hj})\right)^2\\&+\left(1+{1\over \alpha}\right)\E_x^{\theta}\left(\zeta^\theta_{jn}(u)-
\frac{1}{2} r(n)ug_h(\theta;X_{h(j-1)},X_{hj})\right)^2.\end{aligned}
$$
Because by the construction
\be\label{recall}
\sum_{j=1}^n\E_x^{\theta}\left(g_h(\theta;X_{h(j-1)},X_{hj})\right)^2=I_n(\theta)=r^{-2}(n),
\ee
this leads to the bound
$$
\limsup_{n\rightarrow\infty}\sum_{j=1}^{n}\E_x^{\theta}(\zeta^\theta_{jn}(u))^2\leq\frac{1+\alpha}{4}u^2.
$$
Since $\alpha>0$ is arbitrary, this completes the proof.
\end{proof}

\begin{lem}\label{thmp1}  One has  \be\label{T1l2}
\sum_{j=1}^n(\zeta^\theta_{jn}(u))^2 \to
\frac{u^2}{4}, \quad n\to \infty \ee
in $\pr_x^{\theta}$-probability.
\end{lem}

\begin{proof} By the  Chebyshev inequality,
\begin{multline*}
\pr_x^{\theta}\left\{\left|\sum_{j=1}^n(\zeta^\theta_{jn}(u))^2
-\frac{1}{4}r^2(n)u^2\sum_{j=1}^n(g_h(\theta;X_{h(j-1)},X_{hj}))^2\right|>\ve\right\}\\\leq
\frac{1}{\ve}\sum_{j=1}^n\E_x^{\theta}\left|(\zeta^\theta_{jn}(u))^2
-\frac{1}{4}r^2(n)u^2(g_h(\theta;X_{h(j-1)},X_{hj}))^2\right|\\=\frac{1}{\ve}\sum_{j=1}^n\E_x^{\theta}\left|\zeta^\theta_{jn}(u)
-\frac{1}{2}r(n)ug_h(\theta;X_{h(j-1)},X_{hj})\right|\left|\zeta^\theta_{jn}(u)
+\frac{1}{2}r(n)ug_h(\theta;X_{h(j-1)},X_{hj})\right|
\end{multline*}
which by \eqref{ung}, for a given $\alpha>0$, is dominated by
\begin{multline*}
\frac{1}{2\alpha\ve}\sum_{j=1}^n\E_x^{\theta}\left(\zeta^\theta_{jn}(u)
-\frac{1}{2}r(n)ug_h(\theta;X_{h(j-1)},X_{hj})\right)^2\\+
\frac{\alpha}{2\ve}\sum_{j=1}^n\E_x^{\theta}
\left(\zeta^\theta_{jn}(u)+
\frac{1}{2}r(n)ug_h(\theta;X_{h(j-1)},X_{hj})\right)^2.
\end{multline*}
 By \eqref{zbch_eta2} the first item  of
this expression tends to zero as $n\rightarrow\infty$. Furthermore,
the Cauchy inequality together with
\eqref{zbch_eta} and \eqref{recall} imply that for the second one the upper limit does not exceed
$$
\limsup_{n\to \infty}\left(\frac{\alpha}{\ve}\sum_{j=1}^n\E_x^{\theta}(\zeta^\theta_{jn}(u))^2+
\frac{\alpha u^2}{2\ve}r^2(n)\sum_{j=1}^n
\E_x^{\theta}(g_h\left(\theta;X_{h(j-1)},X_{hj}\right))^2 \right)\leq
\frac{3\alpha u^2}{2\ve}.
$$
Since $\alpha>0$ is arbitrary, this proves that the difference
$$\sum_{j=1}^n(\zeta^\theta_{jn}(u))^2
-\frac{1}{4}r^2(n)u^2\sum_{j=1}^n(g_h(\theta;X_{h(j-1)},X_{hj}))^2$$
tends to $0$ in $\pr_x^{\theta}$-probability. Combined with the condition 3 of Theorem \ref{mainthm1}, this gives the required statement.
\end{proof}

\begin{lem}\label{thmp2} One has  \be\label{T1l1}
\max_{1\leq j\leq
n}|\zeta^\theta_{jn}(u)|\to 0, \quad n\to \infty
\ee
in $\pr_x^{\theta}$-probability.
\end{lem}
\begin{proof}
We have
\begin{multline*}
\pr_x^{\theta}\left\{\max_{1\leq j\leq
n}|\zeta^\theta_{jn}(u)|>\ve\right\}\leq
\sum_{j=1}^n\pr_x^{\theta}\left\{|\zeta^\theta_{jn}(u)|>\ve\right\}\\\leq
\sum_{j=1}^n\pr_x^{\theta}\left\{\left|\zeta^\theta_{jn}(u)
-\frac{1}{2}
r(n)ug_h(\theta;X_{h(j-1)},X_{hj})\right|>\frac{\ve}{2}\right\}\\+
\sum_{j=1}^n\pr_x^{\theta}\left\{\left|g_h(\theta;X_{h(j-1)},X_{hj})\right|>
\frac{\ve}{4 r(n)|u|}\right\}.
\end{multline*}
  The first sum in the r.h.s. of this
inequality vanishes as $n\rightarrow\infty$ because of \eqref{zbch_eta2}, the second sum vanishes because of the condition 4 of Theorem \ref{mainthm1}.
\end{proof}

\begin{res}\label{rem} By Lemma \ref{thmp2} and Lemma \ref{thmp1}, we have
\be\label{T1l4}
\sum_{j=1}^n|\zeta^\theta_{jn}(u)|^3\to 0,\quad n\to \infty \ee
in $\pr_x^{\theta}$-probability.
\end{res}

Because of the Markov structure of the sample, in addition to Lemma \ref{thmp1} we will need the following statement. Denote
$$
\mathcal{F}_{j}=\sigma(X_{hi}, i\leq j), \quad \E_{x,j}^{\theta}=\E_x^\theta[\cdot|\mathcal{F}_j].
$$

\begin{lem}\label{lcond} One has \be\label{ll}\sum_{j=1}^n\E_{x,j-1}^{\theta}(\zeta^\theta_{jn}(u))^2 \to
\frac{u^2}{4}, \quad n\to \infty \ee
in $\pr_x^{\theta}$-probability.
\end{lem}
\begin{proof} Denote
$$
\chi_{jn}=(\zeta^\theta_{jn}(u))^2-\E_{x,j-1}^{\theta}(\zeta^\theta_{jn}(u))^2,\quad S_n=\sum_{j=1}^n\chi_{jn},
$$
then by (\ref{T1l2}) it us enough to prove  that $S_n\to 0$ in $\pr_x^{\theta}$-probability.
Fix $\eps>0$, and put
$$
\chi_{jn}^\eps=(\zeta^\theta_{jn}(u))^2 1_{|\zeta^\theta_{jn}(u)|\leq \eps}-\E_{x,j-1}^{\theta}\Big((\zeta^\theta_{jn}(u))^21_{|\zeta^\theta_{jn}(u)|\leq \eps}\Big),\quad S_n^\eps=\sum_{j=1}^n\chi_{jn}^\eps.
$$
By the construction $\{\chi_j^\eps, j=1,\dots, n\}$ is a martingale difference, hence  $$\begin{aligned}
\E_{x}^{\theta}(S_n^\eps)^2&=\sum_{k=1}^n
\E_{x}^{\theta}(\chi_{jn}^\eps)^2\leq \sum_{k=1}^n
\E_{x}^{\theta}(\zeta^\theta_{jn}(u))^4 1_{|\zeta^\theta_{jn}(u)|\leq \eps}\leq \eps^2\E_{x}^{\theta}\sum_{k=1}^n (\zeta^\theta_{jn}(u))^2.
\end{aligned}
$$
Hence by (\ref{zbch_eta}) and the Cauchy inequality,
\be\label{ESn}
\limsup_{n\to\infty}\E_{x}^{\theta}|S_n^\eps|\leq {\eps|u|\over 2}
\ee

Now, let us estimate the difference $S_n-S_n^\eps$. Note that, using the first statement in   Lemma \ref{eta_jn},  one can improve the statement of Lemma \ref{thmp1} and show that the convergence (\ref{T1l2}) holds true in $L_1(\pr_x^\theta)$; see e.g. Theorem A.I.4 in \cite{IKh}. In particular, this means that the sequence
$$
\sum_{j=1}^n(\zeta^\theta_{jn}(u))^2, \quad n\geq 1
$$
is uniformly integrable.  Hence, because by Lemma \ref{thmp2} the probabilities of the sets
\be\label{Oeps}
\Omega_n^\eps=\left\{\max_{j\leq n}|\zeta_{jn}|>\ve\right\}
\ee
 tend to zero as $n\to \infty$, we have
$$
\E_x^\theta\left(1_{\Omega_n^\eps}\sum_{j=1}^n(\zeta^\theta_{jn}(u))^2\right)\to 0.
$$
One has
$$
\chi_{jn}-\chi_{jn}^\eps=(\zeta^\theta_{jn}(u))^2 1_{|\zeta^\theta_{jn}(u)|> \eps}-\E_{x,j}^\theta(\zeta^\theta_{jn}(u))^2 1_{|\zeta^\theta_{jn}(u)|> \eps},
$$
hence
$$\begin{aligned}
\E_x^\theta|S_n-S_n^\eps|&\leq
2\sum_{j=1}^n\E_x^\theta (\zeta^\theta_{jn}(u))^2 1_{|\zeta^\theta_{jn}(u)|> \eps}\leq 2\E_x^\theta\left(1_{\Omega_n^\eps}\sum_{j=1}^n(\zeta^\theta_{jn}(u))^2\right)\to 0.
\end{aligned}$$
Together with (\ref{ESn}) this gives
$$
\limsup_{n\to\infty}\E_{x}^{\theta}|S_n|\leq {\eps|u|\over 2},
$$
which completes the proof because $\eps>0$ is arbitrary.
\end{proof}

The final preparatory result we require is the following.

\begin{lem}\label{thmp3} One has
 \be\label{T1l3}
2\sum_{j=1}^n\zeta^\theta_{jn}(u) -
r(n)u\sum_{j=1}^ng_h(\theta;X_{h(j-1)},X_{hj})\to
-\frac{u^2}{4},\quad n\to \infty \ee in
$\pr_x^{\theta}$-probability.
\end{lem}
\begin{proof} We have the equality
$$(\zeta^\theta_{jn}(u))^2=\frac{p_h(\theta+r(n)u;X_{h(j-1)},X_{hj})}{p_h(\theta;
X_{h(j-1)},X_{hj})} -1-2\zeta^\theta_{jn}(u)$$ valid $\pr_x^\theta$-a.s. Note that by the Markov property of $X$ one has
\begin{align*}
\E_{x,j-1}^\theta&\frac{p_h(\theta+r(n)u;X_{h(j-1)},X_{hj})}{p_h(\theta;
X_{h(j-1)},X_{hj})}\\&=\int_{\mathbb{X}}\frac{p_h(\theta+r(n)u;X_{h(j-1)},y)}{p_h(\theta;
X_{h(j-1)},y)}p_h(\theta;
X_{h(j-1)},y)\lambda(dy)=1;
\end{align*}
hence by Lemma \ref{lcond} one has that
 $$ \sum\limits_{j=1}^n
\E_{x,j-1}^{\theta}\zeta^\theta_{jn}(u) \to -\frac{u^2}{8}$$
in $\pr_x^{\theta}$-probability. Therefore, what we have to prove in fact is that
$$
V_n:=2\sum_{j=1}^n\left(\zeta^\theta_{jn}(u)-
\E_{x,j-1}^{\theta}\zeta^\theta_{jn}(u)\right)
-r(n)u\sum_{j=1}^ng_h(\theta;X_{h(j-1)},X_{hj})\to 0
$$
in $\pr_x^{\theta}$-probability.  By \eqref{g_mart} the sequence
$$
\zeta^\theta_{jn}(u)-
\E_{x,j-1}^{\theta}\zeta^\theta_{jn}(u)
-r(n)ug_h(\theta;X_{h(j-1)},X_{hj}), \quad j=1, \dots n
$$
is a martingale difference, hence
$$
\E_x^\theta V_n^2\leq
4\sum_{j=1}^n\E_x^{\theta}\left(\zeta^\theta_{jn}(u)
-\frac{1}{2}r(n)ug_h(\theta;X_{h(j-1)},X_{hj})\right)^2,
$$
which tends to zero as $n\to \infty$ by \eqref{zbch_eta2}.
\end{proof}

Now, we can finalize the proof of Theorem \ref{mainthm1}. Fix $\eps\in (0,1)$ and consider the sets $\Omega_n^\eps$ defined by  (\ref{Oeps});
by Lemma \ref{thmp2} we have $\pr_x^{\theta}(\Omega_n^\eps)\to 0$. Using the Taylor expansion for the function $\log(1+x)$, we obtain that there exist a constant $C_\eps$ and random variables $\alpha_{jn}$ such that $|\alpha_{jn}|<C_\eps$, for which  the following identity holds true outside of  the set $\Omega_n^\eps$:
$$
\sum_{j=1}^n\log \frac{p_h( \theta+
r(n)u;X_{h(j-1)},X_{hj})}{p_h(\theta;X_{h(j-1)},X_{hj})}=
2\sum_{j=1}^n\zeta^\theta_{jn}(u)-\sum_{j=1}^n(\zeta^\theta_{jn}(u))^2+
\sum_{j=1}^n\alpha_{jn}|\zeta^\theta_{jn}(u)|^3.
$$

 Then by Lemma \ref{thmp1}, Lemma \ref{thmp3}, and Corollary \ref{rem} we have

$$\begin{aligned}
\log Z_{n,\theta}(u)&=\sum_{j=1}^n\log \frac{p_h(\theta+
r(n)u;X_{h(j-1)},X_{hj})}{p_h(\theta;X_{h(j-1)},X_{hj})}\\ &\hspace*{1cm}=r(n)u\sum_{j=1}^ng_h(\theta;X_{h(j-1)},X_{hj})- \frac{u^2}{4}- \frac{u^2}{4}+\Psi_n,
\end{aligned}
$$
where $\Psi_n\to 0$ in $\pr_x^{\theta}$-probability. By the asymptotic normality condition 2, this completes the proof. \qed

\section{Proof of Theorem \ref{mainthm2}}\label{s32}

To prove Theorem \ref{mainthm2}, we verify the conditions of Theorem \ref{mainthm1}. First, let us give an   auxiliary result, which will be used repeatedly in the proof.

\begin{lem} Under conditions \textbf{A} and \textbf{H}  for every $p\in(2,4+\beta)$ there exists a constant $C$ such that for all $x\in\R$, $\theta\in (\theta_-, \theta_+)$,  and $t\geq 0$ \be\label{g_mom} \E_x^{\theta}
\Big|g_h(\theta;x, X_{h})\Big|^p\leq C(1+|x|)^p,\quad \E_x^\theta
|X_{t}|^p\le C(1+|x|^p).\ee
\end{lem}
\begin{proof} The first inequality is proved in Lemma 1 \cite{MLE}. One can prove the second inequality, using a standard argument based on the Lyapunov condition for the function $V(x)=|x|^p;$ e.g. Proposition 4.1 \cite{ergod}.
\end{proof}

Recall (e.g. \cite{ergod}, Section 3.2)  that one standard way to construct the invariant measure $\varkappa^{\theta_0}_{inv}$ is to take a weak limit point (as $T\to \infty$) for the family of \emph{Khas'minskii's averages}
$$
\varkappa^{\theta_0}_{T}(dy)={1\over T}\int_0^T\P_x^{\theta_0}(X_t\in dy)\, dt.
$$
Then, by the Fatou lemma, the second relation in (\ref{g_mom}) implies the following moment  bound for  $\varkappa^{\theta_0}_{inv}$.

\begin{res}\label{cor2} For  every $p\in(2,4+\beta)$,
$$
\int_{\Re}|y|^p\varkappa^{\theta_0}_{inv}(dy)<\infty.
$$
\end{res}

Everywhere below we assume conditions of Theorem \ref{mainthm2} to hold true.

\begin{lem}\label{CPT} The sequence $${1\over \sqrt{n}}
\sum_{j=1}^ng_h\left(\theta_0;X_{h(j-1)},X_{hj}\right), \quad n\geq 1$$
is asymptotically normal w.r.t. $P_x^{\theta_0}$ with parameters $(0,\sigma^2(\theta_0))$, where $\sigma^2(\theta_0)$ is defined in (\ref{sigma}).
\end{lem}

\begin{proof}
The idea of the proof is similar to the one of the proof of  Theorem 3.3
\cite{Leonenko}.  Denote
$$Q_n(\theta_0, X)={1\over \sqrt{n}}
\sum_{j=1}^ng_h\left(\theta_0;X_{h(j-1)},X_{hj}\right).
$$
By Theorem 2.2 \cite{masuda}  (see also Theorem 1.2 \cite{ergod}), the
$\alpha$-mixing coefficient $\alpha(t)$ for the stationary version of the process $X$ does not exceed  $C_3
e^{-C_4t}$, where $C_3,C_4$ are some positive constants.
Then by CLT for stationary sequences (Theorem 18.5.3 \cite{IL}) and
\eqref{g_mom} we have
$$
    Q_n(\theta_0, X^{st,\theta_0})\Rightarrow \mathcal{N}\left(0, \widetilde{\sigma}^2(\theta_0)\right),\quad
    n\rightarrow \infty
$$
with
 $$\widetilde{\sigma}^2(\theta_0)=
\sum_{k=-\infty}^{+\infty}\E\left(
    g_h\left(\theta_0;X_0^{st, \theta_0}, X_{h}^{
    st,\theta_0}\right)
    g_h\left(\theta;X_{h(k-1)}^{ st,\theta_0}, X_{hk}^{
    st,\theta_0}\right)\right).$$
Furthermore, under conditions of Theorem \ref{mainthm2} there
exists an {\it{exponential coupling}} for the process $X$; that is, a  two-component process $Y=(Y^1, Y^2)$, possibly defined on another probability space, such that
$Y^1$ has the distribution $\P_x^{\theta_0}$, $Y^2$ has the same distribution with $X^{st, \theta_0}$, and  for all $t>0$
\begin{equation}\label{ineqcoup}
    \pr\Big(Y_{t}^1\neq Y_{t}^2\Big)\le C_1 e^{-C_2t}
\end{equation}
with some  constants $C_1$, $C_2$. The proof of this fact can be found in \cite{Kulik} (Theorem 2.2).  Then for any Lipschitz continuous function $f:\Re\to \Re$ we have
\begin{multline}\label{zbchCoup}
|\E_x^\theta f(Q_n(\theta_0, X))-\E f(Q_n(\theta_0,
X^{st,\theta_0}))|=|\E f(Q_n(\theta_0, Y^1))-\E f(Q_n(\theta_0,
Y^2))|\\=\mathrm{Lip}(f)\E|Q_n(\theta_0, Y^1)-Q_n(\theta_0, Y^2)|
\\
\leq
\frac{\mathrm{Lip}(f)}{\sqrt{n}}\sum_{k=1}^n\E\left|g_h(\theta_0;Y^1_{h(k-1)},
Y^1_{hk})-g_h(\theta_0; Y^2_{h(k-1)},
Y^2_{hk})\right|1_{(Y^1_{h(k-1)},Y^1_{hk})\neq (Y^2_{h(k-1)},Y^2_{hk})}
\\
 \leq
\frac{2\mathrm{Lip}(f)}{\sqrt{n}}\sum_{k=1}^n\left(\E
\left|g_h\left(\theta_0;Y^1_{h(k-1)},
Y^1_{hk}\right)\right|^p+\E\left|g_h\left(\theta_0;Y^2_{h(k-1)},
Y^2_{hk}\right)\right|^p\right)^{1/p}\\ \times\left(P\Big(Y^1_{h(k-1)}\neq
Y^2_{h(k-1)}\Big)+P\Big(Y^1_{hk}\neq
Y^2_{hk}\Big)\right)^{1/q},
\end{multline}
where $p,q>1$ are such that $1/p+1/q=1$. Since $Y^1$ has the distribution $\P_x^{\theta_0}$, by \eqref{g_mom} we have for $p\in n(2, 4+\beta)$
\begin{multline}\E
\left|g_h\left(\theta_0;Y^1_{h(k-1)},
Y^1_{hk}\right)\right|^p=\E_x^{\theta_0}
\left|g_h\left(\theta_0;X_{h(k-1)},
X_{hk}\right)\right|^p\\\leq C\E_x^{\theta_0}\Big(1+|X_{h(k-1)}|^p)\Big)\leq C+C^2(1+|x|^p).
\end{multline}
Similarly,
\begin{multline}\E
\left|g_h\left(\theta_0;Y^2_{h(k-1)},
Y^2_{hk}\right)\right|^p=\E
\left|g_h\left(\theta_0;X_{h(k-1)}^{st, \theta_0},
X_{hk}^{st, \theta_0}\right)\right|^p\\\leq C\E\Big(1+\Big|X_{h(k-1)}^{st, \theta_0}\Big|^p)\Big)=C+C\int_{\Re}|y|^p\varkappa^{\theta_0}_{inv}(dy),
\end{multline}
and the constant in the right hand side is finite by Corollary \ref{cor2}. Hence  \eqref{ineqcoup} and (\ref{zbchCoup}) yield that
$$
\E_x^\theta f(Q_n(\theta_0, X))\to \E f(\xi),\quad n\to \infty, \quad \xi\sim \mathcal{N}\left(0, \widetilde{\sigma}^2(\theta_0)\right)
$$
for every Lipschitz continuous function $f:\Re\to \Re$. This means that  the sequence $Q_n(\theta_0, X), n\geq 1$
is asymptotically normal w.r.t. $P_x^{\theta_0}$ with parameters $(0,\tilde \sigma^2(\theta_0))$.

To conclude the proof, it remains to show that
$\widetilde{\sigma}^2(\theta_0)={\sigma}^2(\theta_0)$. This follows easily from (\ref{g_mart}) because, by the Markov property of $X^{st,\theta_0}$,  \begin{multline*}
\widetilde{\sigma}^2(\theta_0)=
 \sigma^2(\theta_0)+ 2\sum_{k=1
}^\infty \E\left( g_h\left(\theta_0;X_0^{st, \theta_0}, X_{h}^{
    st,\theta_0}\right)
    g_h\left(\theta;X_{h(k-1)}^{ st,\theta_0}, X_{hk}^{
    st,\theta_0}\right)\right)\\=
\sigma^2(\theta_0)+ 2\sum_{k=1
}^\infty\E\left[g_h(\theta;X_{0}^{st, \theta_0}, X_{h}^{
    st, \theta_0})\Big(\E_x^\theta g_h(\theta_0;x, X_h)\Big)_{x=X_{h(k-1)}^{st, \theta_0}}\right].
\end{multline*} \end{proof}

Similarly, one can prove that $${1\over {n}}
\sum_{j=1}^n\Big(g_h\left(\theta_0;X_{h(j-1)},X_{hj}\right)\Big)^2\to \sigma^2(\theta_0), \quad n\to \infty$$
in $L_1(\pr_x^{\theta_0})$; the argument is completely the same, with the CLT for a stationary sequence replaced by the
Birkhoff-Khinchin ergodic theorem (we omit the details). Hence
\be\label{r_as}
I_n(\theta_0)\sim n\sigma^2(\theta_0), \quad r(n)\sim {1\over \sqrt{n}\sigma(\theta_0)}, \quad n\to \infty.
\ee
Therefore conditions 2 -- 4 of Theorem \ref{mainthm1} are verified. Condition 1 of Theorem \ref{mainthm1} also holds true: regularity property is proved in \cite{MLE}, and positivity of $I_n(\theta)$ follows from (\ref{r_as}).

Let us prove \eqref{loc1th1}, which then would  allow us to apply Theorem \ref{mainthm1}. It is proved in \cite{SDer} that,
under the conditions of Theorem \ref{mainthm2}, the function $q_h(\theta,x,y)$ is $L_2$-differentiable w.r.t. $\theta$, and
$$
\prt_\theta q_h=\frac{1}{2}(\prt_\theta g_h)\sqrt{p_h}+\frac{1}{4}(g_h)^2 \sqrt{p_h}.
$$
In addition, it is proved therein that for every $\gamma\in[1,2+\beta/2)$
 \be\label{dgmom} \E_x^\theta \Big|\prt_\theta
g_h(\theta;x, X_h)\Big|^\gamma\leq C(1+|x|)^\gamma. \ee

Then
$$\begin{aligned}
&\E_x^{\theta}\int_{\R}\left(q_h\left(\theta+ r(n)v,X_{h(j-1)},y\right)-
q_h(\theta,X_{h(j-1)},y)\right)^2 \df y\\&\hspace*{2cm}\le
 r(n)v\E_x^{\theta}\int_{\R}\df
y\int_{0}^{ r(n)v}\left(\prt_\theta
q_h\left(\theta+s,X_{h(j-1)},y\right)\right)^2\df s\\&\hspace*{2cm}\le
\frac{ r(n)v}{4}\E_x^{\theta}\int_{0}^{r(n)v}\df
s\int_{\R}\left(\prt_\theta
g_h\left(\theta+s;X_{h(j-1)},y\right)+\frac{1}{2}g_h\left(\theta+s;X_{h(j-1)},y\right)^2\right)^2\\&\hspace*{7cm}\times
p_h^s(X_{h(j-1)},y)\df y\\&\hspace*{2cm}\le C r(n)^2v^2
\E_x^{\theta}\left(1+(X_{h(j-1)})^{4}\right);
\end{aligned}
$$
in the last inequality we have used \eqref{dgmom} and the first relation  in \eqref{g_mom}.
Using the second relation in \eqref{g_mom}, we get then
$$
\sup_{|v|<N}  r(n)^2\E_x^{\theta} \sum_{j=1}^n\E_x^{\theta}\int_{\R}\left(q_h\left(\theta+ r(n)v,X_{h(j-1)},y\right)-
q_h(\theta,X_{h(j-1)},y)\right)^2 \df y\le CN^2n r(n)^{4}
$$
with a constant $C$ that depends only on $x$. This relation
together with (\ref{r_as}) completes the proof.

\section*{Acknowledgements} The authors are deeply grateful to H. Masuda for a valuable bibliographic help and useful discussion.

\end{document}